\documentclass[12pt]{amsart}
\usepackage{accents}
\usepackage{appendix}
\usepackage{amsfonts}
\usepackage{amsmath}
\usepackage{amssymb}	
\usepackage{amsthm,bm}
\usepackage{array,booktabs,multirow}
\usepackage{braket}
\usepackage{centernot}
\usepackage{cite}
\usepackage{comment}
\usepackage{dsfont}
\usepackage{mathalfa}
\usepackage[shortlabels]{enumitem}
\usepackage{etoolbox}
\usepackage{float}
\usepackage[hang, flushmargin]{footmisc}
\usepackage{latexsym}
\usepackage{lipsum}
\usepackage{needspace}
\usepackage{tikz}
\usepackage{hyperref}
\usetikzlibrary{matrix,arrows}
\usepackage{diagbox}

\makeatletter
\def\namedlabel#1#2{\begingroup
   \def\@currentlabel{#2}%
   \label{#1}\endgroup
}
\makeatother

\theoremstyle{plain}
\newtheorem{thm}{Theorem}[section]
\newtheorem{cor}[thm]{Corollary}
\newtheorem{lem}[thm]{Lemma}
\newtheorem{prop}[thm]{Proposition}

\theoremstyle{definition}

\newtheorem{defn}[thm]{Definition}

\theoremstyle{remark}

\setlist[enumerate,1]{leftmargin=2em}

\def\A{\mathbf A}

\def\C{\mathbb C}

\def\H{\mathcal H}

\def\N{\mathbb N}

\def\T{\mathbf T}

\def\Z{\mathbb Z}

\def\U{U(\mathfrak{sl}_2)}

\newcommand{\floor}[1]{\left\lfloor #1 \right\rfloor}
\newcommand{\ceil}[1]{\left\lceil #1 \right\rceil}

\title[A connection behind the Terwilliger algebras of
 $H(D,2)$ and $\frac{1}{2}H(D,2)$]{A connection behind the Terwilliger algebras \newline  of $H(D,2)$ and $\frac{1}{2}H(D,2)$}

\author{Hau-Wen Huang}
\address{
Department of Mathematics\\
National Central University\\
Chung-Li 32001 Taiwan
}
\email{hauwenh@math.ncu.edu.tw}

\author{Chia-Yi Wen}
\address{
Department of Applied Mathematics\\
National Yang Ming Chiao Tung University\\
Hsinchu 30010 Taiwan
}
\email{cywen.sc08@nycu.edu.tw}

\begin{document}

\begin{abstract}
The universal enveloping algebra $\U$ of $\mathfrak{sl}_2$ is a unital associative algebra over $\C$ generated by $E,F,H$ subject to the relations
\begin{align*}
[H,E]=2E,
\qquad
[H,F]=-2F,
\qquad
[E,F]=H.
\end{align*}
The distinguished central element
$$
\Lambda=EF+FE+\frac{H^2}{2}
$$
is called the Casimir element of $\U$.
The universal Hahn algebra $\H$ is a unital associative algebra over $\C$ with generators $A,B,C$ and the relations assert that
$[A,B]=C$ and each of
\begin{align*}
\alpha=[C,A]+2A^2+B,
\qquad
\beta=[B,C]+4BA+2C
\end{align*}
is central in $\H$.
The distinguished central element
$$
\Omega=4ABA+B^2-C^2-2\beta A+2(1-\alpha)B
$$
is called the Casimir element of $\H$.
By investigating the relationship between the Terwilliger algebras of the hypercube and its halved graph, we discover the algebra homomorphism $\natural:\H\rightarrow \U$ that sends
\begin{eqnarray*}
A &\mapsto & \frac{H}{4}, \\
B & \mapsto & \frac{E^2+F^2+\Lambda-1}{4}-\frac{H^2}{8}, \\
C & \mapsto & \frac{E^2-F^2}{4}.
\end{eqnarray*}
We determine the image of $\natural$ and show that the kernel of $\natural$ is the two-sided ideal of $\H$ generated by $\beta$ and $16 \Omega-24 \alpha+3$.
By pulling back via $\natural$ each $\U$-module can be regarded as an $\H$-module.
For each integer $n\geq 0$ there exists a unique $(n+1)$-dimensional irreducible $\U$-module $L_n$ up to isomorphism.
We show that the $\H$-module $L_n$ ($n\geq 1$) is a direct sum of two non-isomorphic irreducible $\H$-modules.
\end{abstract}

\maketitle



\allowdisplaybreaks

%%%%%%%%%%%%%%%%%%%%%%%%%%%%%%%%%%%%%%%
\section{Introduction}\label{sec:intro}
%%%%%%%%%%%%%%%%%%%%%%%%%%%%%%%%%%%%%%%

Throughout this paper we adopt the following conventions: Let $\Z$ denote the ring of integers. Let $\N$ denote the set of all nonnegative integers. Let $\C$ denote the complex number field. An algebra is meant to be a unital associative algebra. A subalgebra has the same unit as the parent algebra. An algebra homomorphism is meant to be a unital algebra homomorphism. In an algebra the notation $[x,y]$ stands for the commutator
$
xy-yx.
$
Given any nonempty set $X$ let $\C^X$ denote the free vector space over $\C$ generated by $X$; in other words $\C^X$ can be regarded as a vector space over $\C$ that has the basis $X$.

\begin{defn}\label{defn:U}
The {\it universal enveloping algebra $\U$ of $\mathfrak{sl}_2$} is an algebra over $\C$ generated by $E,F,H$ subject to the relations
\begin{align}
& [H,E]=2E, \label{eq:U_relation_E}
\\
& [H,F]=-2F, \label{eq:U_relation_F}
\\
& [E,F]=H. \label{eq:U_relation_H}
\end{align}
\end{defn}
The element
\begin{equation}\label{U_casimir}
  \Lambda=EF+FE+\frac{H^2}{2}
\end{equation}
is central in $\U$ and it is called the {\it Casimir element} of $\U$.

\begin{defn}
\label{defn:H}
The {\it universal Hahn algebra} $\H$ is an algebra over $\C$ generated by $A,B,C$ and the relations assert that
\begin{equation}\label{R0}
  [A,B]=C
\end{equation}
and each of
\begin{align}
& [C,A]+2A^2+B,  \label{R1}
\\
&
[B,C]+4BA+2C \label{R2}
\end{align}
is central in $\H$.
\end{defn}
Let $\alpha$ and $\beta$ denote the elements (\ref{R1}) and (\ref{R2}) of $\H$ respectively.
By \cite[Equation 2.2]{gvz2014} the element
\begin{equation}\label{eq:Casimir_H}
  \Omega=4ABA+B^2-C^2-2\beta A+2(1-\alpha)B
\end{equation}
is central in $\H$ and it is called the {\it Casimir element} of $\H$.

The main result of \cite{Vidunas:2007} implies that the underlying space of the Leonard pair of Hahn type supports a finite-dimensional irreducible $\H$-module.
In \cite[Section 3]{Huang:CG&Johnson} a classification of finite-dimensional irreducible $\H$-modules is provided.
In \cite{Huang:CG&Johnson,Johnson:2021} the algebra $\H$ is connected to the Terwilliger algebra of the Johnson graph. Additionally, there are some applications of $\H$ to physics. Please see  \cite{Hahn:2019-1,Hahn:2019-2} for instance.

By investigating the relationship between the Terwilliger algebras of the hypercube and its halved graph, we find the following connection between $\U$ and $\H$:

\begin{thm}\label{thm:HtoU}
  There exists a unique algebra homomorphism $\natural:\H\rightarrow \U$ that sends
\begin{eqnarray}
A &\mapsto & \frac{H}{4}, \label{A}
\\
B & \mapsto & \frac{E^2+F^2+\Lambda-1}{4}-\frac{H^2}{8},\label{B}\\
C & \mapsto & \frac{E^2-F^2}{4},\label{C}\\
\alpha & \mapsto & \frac{\Lambda-1}{4},\label{alpha}\\
\beta & \mapsto & 0.\label{beta}
\end{eqnarray}
\end{thm}

To be a $\Z$-graded algebra we can grade a free algebra by assigning arbitrary degrees to the generators. We set the degrees of $E,F,H$ to be $1,-1,0$ respectively. Since the elements corresponding to the relations (\ref{eq:U_relation_E})--(\ref{eq:U_relation_H}) are homogeneous of degrees $1,-1,0$, the two-sided ideal generated by these elements is homogeneous. Thus the factor algebra $\U$ inherits the $\Z$-grading.
For each $n\in\Z$ let $U_n$ denote the $n^{\rm th}$ homogeneous subspace of $\U$.
The subspaces $\{U_n\}_{n\in\Z}$ of $\U$ satisfy the following properties:
\begin{description}
  \item[(G1)] $\U=\bigoplus\limits_{n\in\Z}U_n$.
  \item[(G2)] $U_m\cdot U_n\subseteq U_{m+n}$ for all $m,n\in\Z$.
\end{description}

\begin{defn}\label{defn:Ue}
Define
$$
\U_e=\bigoplus\limits_{n\in\Z}U_{2n}.
$$
Since $1\in U_0$ and by {\bf (G2)} the space $\U_e$ is a subalgebra of $\U$.
\end{defn}

We determine the image of $\natural$ and characterize the kernel of $\natural$ as follows:

\begin{thm}\label{thm:image of natural}
\begin{enumerate}
\item ${\rm Im}\,\natural=\U_e$.

\item ${\rm Ker}\, \natural$ is the two-sided ideal of $\H$ generated by $\beta$ and $16 \Omega-24 \alpha+3$.
\end{enumerate}
\end{thm}

By pulling back via $\natural$ each $\U$-module can be regarded as an $\H$-module.  We now recall the finite-dimensional irreducible $\U$-modules.

\begin{lem}\label{lem:Vn_U}
For any $n\in\N$ there exists an $(n+1)$-dimensional $\U$-module $L_n$ satisfying the following conditions:
\begin{enumerate}
\item There exists a basis $\{v_i\}_{i=0}^n$ for $L_n$ such that
\begin{align*}
E v_i&=(n-i+1) v_{i-1}
\quad
(1\leq i\leq n),
\qquad
E v_0=0,
\\
F v_i&=(i+1) v_{i+1}
\quad
(0\leq i\leq n-1),
\qquad
F v_n=0,
\\
H v_i&=(n-2i) v_i
\quad
(0\leq i\leq n).
\end{align*}
\item The element $\Lambda$ acts on $L_n$ as scalar multiplication by $\frac{n(n+2)}{2}$.
\end{enumerate}
\end{lem}

Observe that the $\U$-module $L_n$ $(n\in \N)$ is irreducible. For each $n\in \N$ it is well-known that every $(n+1)$-dimensional irreducible $\U$-module is isomorphic to $L_n$.

\begin{thm}\label{thm:decomposition_Ln}
\begin{enumerate}
\item The $\H$-module $L_0$ is irreducible.

\item For any integer $n\geq 1$ the $\H$-module $L_n$ is a direct sum of two non-isomorphic irreducible $\H$-modules.
\end{enumerate}
\end{thm}

The paper is organized as follows: In \S\ref{sec: HtoU} we prove Theorem \ref{thm:HtoU}. In \S\ref{sec:Ue} we give a presentation for $\U_e$.
In \S\ref{sec:im&ker} we give a proof for Theorem \ref{thm:image of natural}. In \S\ref{sec:Hmodule} we prove Theorem \ref{thm:decomposition_Ln} and classify the finite-dimensional irreducible $\U_e$-modules. In \S\ref{sec:Terwilliger lagebra} we explain the connection of our results with the Terwilliger algebras of the hypercube and its halved graph.

%%%%%%%%%%%%%%%%%%%%%%%%%%%%%%%%%%%%%%%%%%%%%%%%%%%%%%%%%%%
\section{Proof for Theorem \ref{thm:HtoU}}\label{sec: HtoU}
%%%%%%%%%%%%%%%%%%%%%%%%%%%%%%%%%%%%%%%%%%%%%%%%%%%%%%%%%%%

\begin{lem}\label{lem:homomorphism_rho}
There exists a unique algebra homomorphism $\rho:\U\rightarrow\U$ that sends
  \begin{eqnarray*}
  % \nonumber % Remove numbering (before each equation)
    (E,F,H) &\mapsto & (F,E,-H).
  \end{eqnarray*}
Moreover $\rho$ satisfies the following properties:
\begin{enumerate}
\item $\rho^2=1$.

\item $\rho(\Lambda)=\Lambda$.

\item $\rho(U_n)=U_{-n}$ for each $n\in\Z$.
\end{enumerate}
\end{lem}
\begin{proof}
It is routine to verify the existence of $\rho$ by using Definition~\ref{defn:U}. Since the algebra $\U$ is generated by $E,F,H$ the uniqueness of $\rho$ follows.

 (i): The algebra homomorphism $\rho^2$ fixes $E,F,H$.

 (ii): Apply $\rho$ to either side of (\ref{U_casimir}).

 (iii): Observe that
 \begin{gather}\label{eq:rho(Un)}
 \rho(U_n)\subseteq U_{-n}
 \qquad\hbox{for all $n\in\Z$}.
 \end{gather}
 Applying $\rho$ to either side of (\ref{eq:rho(Un)}) it follows from Lemma~\ref{lem:homomorphism_rho}(i) that
\begin{equation}\label{eq:rho(Un)'}
U_n\subseteq \rho(U_{-n})
\qquad\hbox{for all $n\in\Z$}.
\end{equation}
Lemma \ref{lem:homomorphism_rho}(iii) follows from (\ref{eq:rho(Un)}) and (\ref{eq:rho(Un)'}).
\end{proof}

Observe that
\begin{equation}\label{eq:U_identity}
  [xy,z]=x[y,z]+[x,z]y
\end{equation}
for any elements $x,y,z$ in an algebra.

\begin{lem}\label{lem:HE^n,HF^n}
For each $n\in \N$ the following equations hold in $\U$:
  \begin{enumerate}
    \item $[H,E^n]=2nE^n$.
    \item $[H,F^n]=-2nF^n$.
  \end{enumerate}
\end{lem}
\begin{proof}
 (i): We proceed by induction on $n$. There is nothing to prove when $n=0$. Suppose that $n\geq 1$.
        Applying (\ref{eq:U_identity}) with $(x,y,z)=(E,E^{n-1},H)$ yields that
        $$
        [E^n,H]=E[E^{n-1},H]+[E,H]E^{n-1}.
        $$
Lemma~\ref{lem:HE^n,HF^n}(i) follows by applying the induction hypothesis and (\ref{eq:U_relation_E}) to the above equation.

(ii): Lemma~\ref{lem:HE^n,HF^n}(ii) follows by applying $\rho$ to Lemma~\ref{lem:HE^n,HF^n}(i) and evaluating the resulting equation by Lemma \ref{lem:homomorphism_rho}.
\end{proof}

\begin{lem}\label{lem:H^2E^n,H^2F^n}
For each $n\in \N$ the following equations hold in $\U$:
  \begin{enumerate}
    \item $[H^2,E^n]=4n(H-n)E^n$.
    \item $[H^2,F^n]=-4n(H+n)F^n$.
  \end{enumerate}
\end{lem}
\begin{proof}
(i): Lemma~\ref{lem:H^2E^n,H^2F^n}(i) follows by applying (\ref{eq:U_identity}) with $(x,y,z)=(H,H,E^n)$ and simplifying the resulting equation by using Lemma \ref{lem:HE^n,HF^n}(i).

(ii): Lemma~\ref{lem:H^2E^n,H^2F^n}(ii) follows by applying $\rho$ to Lemma~\ref{lem:H^2E^n,H^2F^n}(i) and evaluating the resulting equation by Lemma \ref{lem:homomorphism_rho}.
\end{proof}

\begin{lem}\label{lem:E^nF^n,F^nE^n}
For each integer $n\geq 1$ the following equations hold in $\U$:
  \begin{enumerate}
    \item $E^nF^n=\displaystyle\prod\limits_{i=1}^n\frac{2\Lambda-(H-2i+2)(H-2i)}{4}$.
    \item $F^nE^n=\displaystyle\prod\limits_{i=1}^n\frac{2\Lambda-(H+2i-2)(H+2i)}{4}$.
  \end{enumerate}
\end{lem}
\begin{proof}
 (i): We proceed by induction on $n$. By (\ref{eq:U_relation_H}) we have $FE=EF-H$. Substituting this into (\ref{U_casimir}) yields that
Lemma~\ref{lem:E^nF^n,F^nE^n}(i) holds for $n=1$.  Suppose that $n\geq 2$.
Observe that
        \begin{equation}\label{eq:E^nF^n}
        E^nF=E^{n-1}(EF)=E^{n-1}\left(\frac{\Lambda +H}{2}-\frac{H^2}{4}\right).
        \end{equation}
        Applying Lemmas~\ref{lem:HE^n,HF^n}(i) and \ref{lem:H^2E^n,H^2F^n}(i) to (\ref{eq:E^nF^n}) yields that
        $$
        E^nF=\frac{2\Lambda-(H-2n+2)(H-2n)}{4}E^{n-1}.
        $$
Lemma~\ref{lem:E^nF^n,F^nE^n}(i) follows by right multiplying the above equation by $F^{n-1}$ and applying the induction hypothesis to the resulting equation.

 (ii): Lemma~\ref{lem:E^nF^n,F^nE^n}(ii) follows by applying $\rho$ to Lemma~\ref{lem:E^nF^n,F^nE^n}(i) and evaluating the resulting equation by Lemma \ref{lem:homomorphism_rho}.
\end{proof}

\noindent {\it Proof of Theorem \ref{thm:HtoU}.}
Let $A^{\natural},B^{\natural},C^{\natural},{\alpha}^{\natural},{\beta}^{\natural}$ denote the right-hand sides of (\ref{A})--(\ref{beta}) respectively. Observe that
\begin{align}
   & [A^{\natural},B^{\natural}]=\frac{[H,E^2+F^2]}{16}, \label{eq:A'B'} \\
   & [C^{\natural},A^{\natural}]=\frac{[H,F^2-E^2]}{16}, \label{eq:A'C'} \\
   & [B^{\natural},C^{\natural}]=\frac{[F^2,E^2]}{8}+\frac{[H^2,F^2-E^2]}{32}. \label{eq:C'B'}
\end{align}
Applying Lemma \ref{lem:HE^n,HF^n} to (\ref{eq:A'B'}) and (\ref{eq:A'C'}) yields that
\begin{align*}
[A^{\natural},B^{\natural}]
&=\frac{E^2-F^2}{4}
=C^\natural,
\\
[C^{\natural},A^{\natural}]
&=-\frac{E^2+F^2}{4}=-2{A^{\natural}}^2-B^{\natural}+\alpha^\natural.
\end{align*}
Applying Lemmas \ref{lem:HE^n,HF^n}--\ref{lem:E^nF^n,F^nE^n} to (\ref{eq:C'B'}) yields that
\begin{align*}
[B^{\natural},C^{\natural}]
=\frac{(1-E^2-F^2-\Lambda) H}{4}
+\frac{H^3}{8}
+\frac{F^2-E^2}{2}
=-4B^{\natural}A^{\natural}-2C^{\natural}+\beta^\natural.
\end{align*}
By Definition \ref{defn:H} the existence of $\natural$ follows. Since the algebra $\H$ is generated by $A,B,C$ the uniqueness of $\natural$ follows.
\hfill $\square$

%%%%%%%%%%%%%%%%%%%%%%%%%%%%%%%%%%%%%%%%%%%%%%%%%%%%%%%%%%%%%%%
\section{A presentation for $\U_e$}\label{sec:Ue}
%%%%%%%%%%%%%%%%%%%%%%%%%%%%%%%%%%%%%%%%%%%%%%%%%%%%%%%%%%%%%%%

\begin{lem}\label{lem:Un_basis_Poincare}
  For each $n\in\Z$ the elements
  $$
  E^iF^jH^k\qquad\hbox{for all $i,j,k\in\N$ with $i-j=n$}
  $$
  are a basis for $U_n$.
\end{lem}
\begin{proof}
  Immediate from Poincar\'{e}--Birkhoff--Witt theorem.
\end{proof}

\begin{lem}\label{lem:Un_basis}
For all $n\in\N$ the following hold:
\begin{enumerate}
  \item The elements
$$
E^{n}\Lambda^iH^k\qquad\hbox{for all $i,k\in\N$}
$$
are a basis for $U_n$.

  \item The elements
$$
F^{n}\Lambda^iH^k\qquad\hbox{for all $i,k\in\N$}
$$
are a basis for $U_{-n}$.
\end{enumerate}
\end{lem}
\begin{proof}
(i): Let $n\in\N$ be given. For each $i\in \N$ let $U_n^{(i)}$ denote the subspace of $U_n$ spanned by
$E^{n+j} F^j H^k$ for all $j,k\in \N$ with $j\leq i$.
To see Lemma \ref{lem:Un_basis}(i)  it suffices to show that $U_n^{(i)}$ $(i\in \N)$ has the basis
\begin{gather}\label{U_n^i:basis}
E^n \Lambda^j H^k
\qquad
\hbox{for all $j,k\in \N$ with $j\leq i$}.
\end{gather}
To see this we proceed by induction on $i$. There is nothing to prove when $i=0$. Suppose that $i\geq 1$. By construction the quotient space $U_n^{(i)}/U_n^{(i-1)}$ has the basis
$
E^{n+i}F^i H^k+U_n^{(i-1)}
$
for all $k\in \N$.
It follows from Lemma \ref{lem:E^nF^n,F^nE^n}(i) and the induction hypothesis that
\begin{gather*}
E^{n}\Lambda^i H^k+U_n^{(i-1)}=2^i E^{n+i}F^i H^k+U_n^{(i-1)}
\qquad
\hbox{for all $k\in \N$}
\end{gather*}
are a basis for $U_n^{(i)}/U_n^{(i-1)}$.
Combined with the induction hypothesis the elements (\ref{U_n^i:basis}) form a basis for $U_n^{(i)}$.

(ii):
Let $n\in\N$ be given.
By Lemma \ref{lem:homomorphism_rho} the algebra homomorphism $\rho$ maps $E^n\Lambda^iH^k$ to $(-1)^kF^n\Lambda^iH^k$ for all $i,k\in\N$.
From Lemma \ref{lem:homomorphism_rho}(i), (iii)  we see that the map $\rho|_{U_n}:U_n\to U_{-n}$ is a linear isomorphism.
Lemma~\ref{lem:Un_basis}(ii) follows by Lemma~\ref{lem:Un_basis}(i) and the above comments.
\end{proof}

Recall the subalgebra $\U_e$ of $\U$ from Definition~\ref{defn:Ue}.

\begin{lem}\label{lem:Ue_basis}
The elements
  \begin{align}
   E^{2n}\Lambda^i H^k \qquad &\hbox{for all integers $n\geq 1$ and all $i,k\in\N$};
   \label{eq:Ue_basis_1}
   \\
  \Lambda^i H^k \qquad &\hbox{for all $i,k\in \N$};
  \label{eq:Ue_basis_2}
  \\
  F^{2n}\Lambda^iH^k\qquad& \hbox{for all integers $n\geq 1$ and all $i,k\in\N$}
  \label{eq:Ue_basis_3}
  \end{align}
  form a basis for $\U_e$.
\end{lem}
\begin{proof}
Immediate from Lemma~\ref{lem:Un_basis}.
\end{proof}

\begin{thm}\label{thm:Ue_generators}
  The algebra $\U_e$ has a presentation given by generators $E^2,F^2,\Lambda,H$ and relations
  \begin{align}
      [H,E^2]&=4E^2,
     \label{Ue1}\\
      [H,F^2]&=-4F^2,
     \label{Ue2}\\
      16 E^2F^2&=(H^2-2H-2\Lambda)(H^2-6H-2\Lambda+8),
      \label{Ue3} \\
      16 F^2E^2&=(H^2+2H-2\Lambda)(H^2+6H-2\Lambda+8),
     \label{Ue4}\\
      \Lambda E^2&=E^2\Lambda,\qquad \Lambda F^2=F^2\Lambda,\qquad \Lambda H=H\Lambda.
     \label{Ue5}
  \end{align}
\end{thm}
\begin{proof}
Let $\U_e^{'}$ denote the algebra generated by the symbols $E^2,F^2,\Lambda,H$ subject to the relations (\ref{Ue1})--(\ref{Ue5}). By Lemmas \ref{lem:HE^n,HF^n}, \ref{lem:E^nF^n,F^nE^n} and since $\Lambda$ is central in $\U$,
the algebra $\U_e$ satisfies the relations (\ref{Ue1})--(\ref{Ue5}). Thus there exists a unique algebra homomorphism $\Phi:\U_e^{'}\to \U_e$ that sends
\begin{eqnarray*}
(E^2,F^2,\Lambda,H)
&\mapsto&
(E^2,F^2,\Lambda,H).
\end{eqnarray*}
By Lemma \ref{lem:Ue_basis} the elements (\ref{eq:Ue_basis_1})--(\ref{eq:Ue_basis_3}) of $\U_e^{'}$ are linearly independent.
To prove that $\Phi$ is an isomorphism, it remains to show the elements (\ref{eq:Ue_basis_1})--(\ref{eq:Ue_basis_3}) of $\U_e^{'}$ span $\U_e^{'}$.

Let $\U_{e,+}^{'}$ denote the subspace of $\U_{e}^{'}$ spanned by (\ref{eq:Ue_basis_1}).
Let $\U_{e,0}^{'}$ denote the subspace of $\U_{e}^{'}$ spanned by (\ref{eq:Ue_basis_2}).
Let $\U_{e,-}^{'}$ denote the subspace of $\U_{e}^{'}$ spanned by (\ref{eq:Ue_basis_3}). Observe that
\begin{align}
E^2 \U_{e,+}^{'} &\subseteq \U_{e,+}^{'},
\qquad
E^2 \U_{e,0}^{'} \subseteq \U_{e,+}^{'},
\label{E2Ue+}
\\
F^2 \U_{e,-}^{'} &\subseteq \U_{e,-}^{'},
\qquad
F^2 \U_{e,0}^{'} \subseteq \U_{e,-}^{'}.
\label{F2Ue-}
\end{align}
It follows from (\ref{Ue5}) that
\begin{gather}
H  \U_{e,0}^{'} \subseteq \U_{e,0}^{'},
\label{HUe0}
\\
\Lambda \U_{e,+}^{'} \subseteq \U_{e,+}^{'},
\quad
\Lambda \U_{e,0}^{'} \subseteq \U_{e,0}^{'},
\quad
\Lambda \U_{e,-}^{'} \subseteq \U_{e,-}^{'}.
\label{LUe}
\end{gather}
We claim that
\begin{align}
H \U_{e,+}^{'} &\subseteq \U_{e,+}^{'},
\label{HUe+}
\\
H \U_{e,-}^{'} &\subseteq \U_{e,-}^{'},
\label{HUe-}
\\
E^2  \U_{e,-}^{'} &\subseteq \U_{e,0}^{'}+\U_{e,-}^{'},
\label{E2Ue-}
\\
F^2  \U_{e,+}^{'} &\subseteq \U_{e,0}^{'}+\U_{e,+}^{'}.
\label{F2Ue+}
\end{align}
To see (\ref{HUe+}) it suffices to show that
\begin{gather}\label{HUe}
HE^{2n}\Lambda^i H^k\in \U_{e,+}^{'}
\end{gather}
for all integers $n\geq 1$ and all $i,k\in \N$.
We proceed by induction on $n$. Using (\ref{Ue1}) yields that the left-hand side of (\ref{HUe}) is equal to
$$
E^2HE^{2n-2}\Lambda^i H^k+4E^{2n}\Lambda^i H^k.
$$
Clearly the second summand is in $\U_{e,+}^{'}$.
If $n=1$ then the first summand $E^2H\Lambda^i H^k\in \U_{e,+}^{'}$ by (\ref{Ue5}). Hence (\ref{HUe}) holds for $n=1$. If $n\geq 2$ then the first summand is in $\U_{e,+}^{'}$ by (\ref{E2Ue+}) and the induction hypothesis. Therefore (\ref{HUe+}) follows. By a similar argument (\ref{HUe-}) follows.
We now show (\ref{E2Ue-}).
Let $n\geq 1$ be an integer and $i,k\in \N$.
Using (\ref{Ue3}) yields that $E^2 F^{2n}\Lambda^i H^k$ is equal to
\begin{gather}\label{E2F2LH}
\frac{(H^2-2H-2\Lambda)(H^2-6H-2\Lambda+8)}{16}F^{2n-2}\Lambda^i H^k.
\end{gather}
If $n=1$ then (\ref{E2F2LH}) is in $\U_{e,0}^{'}$ by (\ref{Ue5}). If $n\geq 2$ then (\ref{E2F2LH}) is in $\U_{e,-}^{'}$ by (\ref{LUe}) and (\ref{HUe-}). Therefore (\ref{E2Ue-}) holds. By a similar argument (\ref{F2Ue+}) follows. By (\ref{E2Ue+})--(\ref{F2Ue+}) the space $\U_{e,+}^{'}+\U_{e,0}^{'}+\U_{e,-}^{'}$  is a left ideal of $\U_{e}^{'}$. Since $1\in \U_{e,0}^{'}$ it follows that $\U_{e,+}^{'}+\U_{e,0}^{'}+\U_{e,-}^{'}=\U_{e}^{'}$. In other words $\U_{e}^{'}$ is spanned by (\ref{eq:Ue_basis_1})--(\ref{eq:Ue_basis_3}). The result follows.
\end{proof}

%%%%%%%%%%%%%%%%%%%%%%%%%%%%%%%%%%%%%%%%%%%%%%%%%%%%%%%%%%%%%%%%%%%%%%%%
\section{Proof for Theorem \ref{thm:image of natural}}\label{sec:im&ker}
%%%%%%%%%%%%%%%%%%%%%%%%%%%%%%%%%%%%%%%%%%%%%%%%%%%%%%%%%%%%%%%%%%%%%%%%

Define
\begin{align}
\widehat{E^2}&=4A^2+2B+2C-2\alpha,
\label{eq:natural_preimage_E^2}
\\
\widehat{F^2}&=4A^2+2B-2C-2\alpha,
 \label{eq:natural_preimage_F^2}
\\
\widehat{\Lambda}&=1+4\alpha,
\label{eq:natural_preimage_Lambda}
\\
\widehat{H}&=4A.
\label{eq:natural_preimage_H}
\end{align}

\begin{lem}
\label{lem:bar(E2,F2,L,H)}
The images of
$\widehat{E^2},\widehat{F^2},\widehat{\Lambda},\widehat{H}$
under $\natural$ are equal to $E^2,F^2,\Lambda,H$ respectively.
\end{lem}
\begin{proof}
It is routine to verify the lemma by using Theorem \ref{thm:HtoU}.
\end{proof}

For convenience we use the notation $x^\natural$ to denote the image of $x$ under $\natural$ for all $x\in\H$.

\begin{lem}\label{lem:image_A,B,C,alpha,Omega}
\begin{enumerate}
  \item $A^\natural\in U_0$.
  \item $B^\natural\in U_2\oplus U_0\oplus U_{-2}$.
  \item $C^\natural\in U_2\oplus U_{-2}$.
  \item $\alpha^\natural\in U_0$.
\end{enumerate}
\end{lem}
\begin{proof}
(i), (iii): Immediate from (\ref{A}) and (\ref{C}).

(ii), (iv): Immediate from (\ref{B}), (\ref{alpha}) and the fact that $\Lambda\in U_0$ by (\ref{U_casimir}).
\end{proof}

The proof of Theorem \ref{thm:image of natural}(i) is as follows:

\medskip

\noindent {\it Proof of Theorem \ref{thm:image of natural}(i).}
By Definition \ref{defn:H} the algebra $\H$ is generated by $A,B,C$.
Combined with Lemma~\ref{lem:image_A,B,C,alpha,Omega}(i)--(iii) this implies that ${\rm Im}\,\natural \subseteq \U_{e}$.
By Theorem~\ref{thm:Ue_generators} the algebra $\U_{e}$ is generated by $E^2,F^2,\Lambda,H$. Combined with Lemma \ref{lem:bar(E2,F2,L,H)} this implies that $\U_{e}\subseteq {\rm Im}\,\natural $. Therefore Theorem \ref{thm:image of natural}(i) follows.
\hfill $\square$

\begin{lem}\label{lem:tilde_rho}
  There exists a unique algebra homomorphism $\widetilde{\rho}:\H\rightarrow\H$ that sends
  \begin{eqnarray*}
  (A,B,C,\alpha,\beta)&\mapsto& (-A,B,-C,\alpha,-\beta).
  \end{eqnarray*}
Moreover $\widetilde{\rho}$ satisfies the following properties:
\begin{enumerate}
    \item $\widetilde{\rho}^2=1$.
    \item $\widetilde{\rho}(\Omega)=\Omega$.
    \item $\natural\circ\widetilde{\rho}=\rho\circ\natural$.
\end{enumerate}
\end{lem}
\begin{proof}
It is routine to verify the existence of $\widetilde{\rho}$ by using Definition~\ref{defn:H}. Since the algebra $\H$ is generated by $A,B,C$ the uniqueness of $\widetilde{\rho}$ follows.

(i): The homomorphism ${\widetilde{\rho}}^2$ fixes $A,B,C$.

 (ii): Apply $\widetilde{\rho}$ to either side of (\ref{eq:Casimir_H}).

 (iii): By Theorem~\ref{thm:HtoU} and Lemma~\ref{lem:homomorphism_rho} the algebra homomorphisms $\natural\circ\widetilde{\rho}$ and $\rho\circ\natural$ agree at $A,B,C$.
\end{proof}

\begin{lem}\label{lem:Omega}
The following equations hold in $\H$:
\begin{enumerate}
\item $\displaystyle \frac{\Omega-B^2+C^2}{2}+\beta A=2BA^2+2CA+(1-\alpha)B$.

\item $\displaystyle \frac{\Omega-B^2+C^2}{2}+\beta A=A^2B+BA^2-2A^2-\alpha B+\alpha$.
\end{enumerate}
\end{lem}
\begin{proof}
(i): By (\ref{R0}) we have $ABA=(C+BA)A$. Lemma \ref{lem:Omega}(i) follows by substituting this into (\ref{eq:Casimir_H}).

(ii): The subtraction of the right-hand side of Lemma \ref{lem:Omega}(ii) from the right-hand side of Lemma \ref{lem:Omega}(i) is equal to
\begin{gather}\label{Omega1-Omega2}
[B,A^2]+2A^2+2CA+B-\alpha.
\end{gather}
Applying (\ref{eq:U_identity}) with $(x,y,z)=(A,A,B)$ yields that
$[A^2,B]=AC+CA$. Hence (\ref{Omega1-Omega2}) is equal to
$[C,A]+2A^2+B-\alpha=0$ by the setting (\ref{R1}) of $\alpha$. Therefore Lemma \ref{lem:Omega}(ii) follows.
\end{proof}

\begin{lem}\label{lem:image_omega}
\begin{enumerate}
 \item $\Omega^\natural\in U_4\oplus U_2\oplus U_0\oplus U_{-2}\oplus U_{-4}$.

\item $ \Omega^\natural= \frac{3}{16}(2\Lambda-3)$.
\end{enumerate}
\end{lem}
\begin{proof}
(i): Recall from (\ref{beta}) that $\beta^\natural=0$. Combined with Lemma \ref{lem:image_A,B,C,alpha,Omega} this yields Lemma \ref{lem:image_omega}(i).

(ii): Let $x\in\U$ be given. By {\bf (G1)} there are unique $x_n\in U_n$ for all $n\in\Z$ such that
$$
x=\sum\limits_{n\in\Z}x_n.
$$
The element $x_n$ is called the $n^{\rm th}$ homogeneous component of $x$ for each $n\in\Z$.
Using Lemma \ref{lem:Omega}(i) yields that
\begin{align*}
  & \Omega_4^\natural = {(B_2^{\natural})}^2-{(C_2^{\natural})}^2,
  \\
  & \Omega_2^\natural =
  4(C_2^{\natural}+B_2^{\natural}A_0^{\natural}) A_0^{\natural}
  +
  B_0^{\natural}B_2^{\natural}+B_2^{\natural}B_0^{\natural}
  +2(1-\alpha_0^{\natural})B_2^{\natural},
  \\
  & \Omega_0^\natural =
  4B_0^{\natural}{(A_0^{\natural})}^2
  + {(B_0^{\natural})}^2+B_2^{\natural}B_{-2}^{\natural}
  +B_{-2}^{\natural}B_2^{\natural}
  -C_2^{\natural}C_{-2}^{\natural}
  -C_{-2}^{\natural}C_2^{\natural}
  +2(1-\alpha_0^{\natural})B_0^{\natural}.
\end{align*}

In the table below we list the nonzero homogeneous components of $A^{\natural},B^{\natural},C^{\natural},\alpha^{\natural}$:
\begin{table}[H]
\renewcommand\arraystretch{1.5}
\centering
\begin{tabular}{c|cccc}

  $x$ & $A$ & $B$ & $C$ & $\alpha$ \\
  \hline
  $x_2^{\natural}$ & 0 & $\frac{E^2}{4}$ & $\frac{E^2}{4}$ & 0 \\
  $x_0^{\natural}$ & $\frac{H}{4}$ & $\frac{\Lambda-1}{4}-\frac{H^2}{8}$ & 0 & $\frac{\Lambda-1}{4}$ \\
  $x_{-2}^{\natural}$ & 0 & $\frac{F^2}{4}$ &  $-\frac{F^2}{4}$ & 0 \\
\end{tabular}
\end{table}
\noindent A direct calculation shows that $\Omega_4^\natural=0$ and
\begin{align}
  \Omega_2^\natural &= \frac{E^2}{2}+\frac{E^2H}{4}-\frac{[H^2,E^2]}{32}, \label{eq:omega_2'}\\
  \Omega_0^\natural &= \frac{E^2F^2+F^2E^2}{8}-\frac{(H^2-2\Lambda+2)(H^2-2\Lambda+18)}{64}. \label{eq:omega_0'}
\end{align}
Applying Lemmas~\ref{lem:HE^n,HF^n}(i) and \ref{lem:H^2E^n,H^2F^n}(i) yields that (\ref{eq:omega_2'}) is equal to $0$. Applying Lemma~\ref{lem:E^nF^n,F^nE^n} yields that (\ref{eq:omega_0'}) is equal to $\frac{3}{16}(2\Lambda-3)$. Applying $\natural$ to either side of Lemma~\ref{lem:tilde_rho}(ii), it follows from Lemma~\ref{lem:tilde_rho}(iii) that
$
\rho(\Omega^{\natural})=\Omega^{\natural}$.
Combined with Lemma~\ref{lem:homomorphism_rho}(iii) this yields that
\begin{gather*}
\rho(\Omega_n^{\natural})=\Omega_{-n}^{\natural}
\qquad
\hbox{for all $n\in\Z$}.
\end{gather*}
Hence $\Omega_{-4}^{\natural}=0$ and $\Omega_{-2}^{\natural}=0$. Lemma \ref{lem:image_omega}(ii) follows from the above comments and Lemma~\ref{lem:image_omega}(i).
\end{proof}

Recall the elements
$
\widehat{E^2},
\widehat{F^2},
\widehat{\Lambda},
\widehat{H}
$
of $\H$ from (\ref{eq:natural_preimage_E^2})--(\ref{eq:natural_preimage_H}).

\begin{lem}\label{lem:E2F2 tilde_rho}
The algebra homomorphism $\widetilde{\rho}:\H\to \H$ sends
\begin{eqnarray*}
(\widehat{E^2},\widehat{F^2},\widehat{\Lambda},\widehat{H})
&\mapsto &
(\widehat{F^2},\widehat{E^2},\widehat{\Lambda},-\widehat{H}).
\end{eqnarray*}
\end{lem}
\begin{proof}
It is routine to verify the lemma by using Lemma \ref{lem:tilde_rho}.
\end{proof}

\begin{lem}\label{lem:bar(HE^2,HF^2)}
The following equations hold in $\H$:
\begin{enumerate}
  \item $[\widehat{H},\widehat{E^2}]=4\widehat{E^2}$.
  \item $[\widehat{H},\widehat{F^2}]=-4\widehat{F^2}$.
\end{enumerate}
\end{lem}
\begin{proof}
 (i): Observe that
$[\widehat{H},\widehat{E^2}]=8[A,B+C]$.
By (\ref{R0}) and (\ref{R1}) it is equal to $4\widehat{E^2}$.

 (ii): Lemma \ref{lem:bar(HE^2,HF^2)}(ii) follows by applying $\widetilde{\rho}$ to Lemma \ref{lem:bar(HE^2,HF^2)}(i) and evaluating the resulting equation by Lemma \ref{lem:E2F2 tilde_rho}.
\end{proof}

\begin{lem}\label{lem:pre_bar(E^2F^2,F^2E^2)}
The following equations hold in $\H$:
\begin{enumerate}
\item $[A,C]=2A^2+B-\alpha$.

\item $[A^2,C]=4 A^3+2AB-2\alpha A-C$.

\item $[[A,C],C]= 8 A^3-4\alpha A+\beta$.
\end{enumerate}
\end{lem}
\begin{proof}
(i): Immediate from the setting (\ref{R1}) of $\alpha$.

(ii): Applying (\ref{eq:U_identity}) with $(x,y,z)=(A,A,C)$ yields that
$$
[A^2,C]=A[A,C]+[A,C]A.
$$
Substituting Lemma \ref{lem:pre_bar(E^2F^2,F^2E^2)}(i) into the right-hand side of the above equation yields that
$$
[A^2,C]=4A^3+AB+BA-2\alpha A.
$$
By (\ref{R0}) we have $BA=AB-C$. Substituting this into the above equation yields Lemma \ref{lem:pre_bar(E^2F^2,F^2E^2)}(ii).

(iii): By Lemma \ref{lem:pre_bar(E^2F^2,F^2E^2)}(i) the commutator
$$
[[A,C],C]=2[A^2,C]+[B,C].
$$
Lemma \ref{lem:pre_bar(E^2F^2,F^2E^2)}(iii) follows by evaluating the right-hand side of the above equation by using (\ref{R0}), the setting (\ref{R2}) of $\beta$ and Lemma \ref{lem:pre_bar(E^2F^2,F^2E^2)}(ii).
\end{proof}

\begin{lem}\label{lem:bar(E^2F^2,F^2E^2)}
The following equations hold in $\H$:
\begin{enumerate}
\item $16 \widehat{E^2}\widehat{F^2}
  -
(\widehat{H}^2-2\widehat{H}-2\widehat{\Lambda})(\widehat{H}^2-6\widehat{H}-2\widehat{\Lambda}+8)
=4(16 \Omega-24\alpha+3)+64 \beta (2A-1)$.

\item $16 \widehat{F^2}\widehat{E^2}-
(\widehat{H}^2+2\widehat{H}-2\widehat{\Lambda})(\widehat{H}^2+6\widehat{H}-2\widehat{\Lambda}+8)
=4(16 \Omega-24\alpha+3)+64 \beta (2A+1)$.
\end{enumerate}
\end{lem}
\begin{proof}
(i): Using the setting (\ref{R1}) of $\alpha$ yields that $\widehat{E^2}$ and $\widehat{F^2}$ are equal to
$2([A,C]+C)$
and
$2([A,C]-C)$ respectively. Hence
$\widehat{E^2}\widehat{F^2}$ is equal to $4$ times
\begin{align}\label{bar(EF)}
[A,C]^2-[[A,C],C]-C^2.
\end{align}
By Lemma \ref{lem:pre_bar(E^2F^2,F^2E^2)}(i), (iii) the element (\ref{bar(EF)}) is equal to
\begin{align*}
4A^4-8A^3+B^2-C^2+2 A^2B+2BA^2-4\alpha A^2+4\alpha A-2 \alpha B+\alpha^2-\beta.
\end{align*}
By (\ref{eq:natural_preimage_Lambda}) and (\ref{eq:natural_preimage_H}) the product $(\widehat{H}^2-2\widehat{H}-2\widehat{\Lambda})(\widehat{H}^2-6\widehat{H}-2\widehat{\Lambda}+8)$ is equal to $4$ times
\begin{align*}
64 A^4-128 A^3+64 A^2
-64 \alpha A^2+64 \alpha A+16 \alpha^2-8 \alpha-3.
\end{align*}
Using the above results yields that the left-hand side of Lemma \ref{lem:bar(E^2F^2,F^2E^2)}(i) is equal to $4$ times
\begin{align}\label{relation:bar(E^2F^2)}
16B^2-16 C^2+ 32 A^2B+ 32 BA^2-64 A^2-32 \alpha B+ 8\alpha-16 \beta + 3.
\end{align}
By Lemma \ref{lem:Omega}(ii) the element (\ref{relation:bar(E^2F^2)}) is equal to
$16 \Omega-24 \alpha+3+16\beta(2A-1)$. Therefore Lemma \ref{lem:bar(E^2F^2,F^2E^2)}(i) follows.

 (ii):
Lemma \ref{lem:bar(E^2F^2,F^2E^2)}(ii) follows by applying $\widetilde{\rho}$ to  Lemma \ref{lem:bar(E^2F^2,F^2E^2)}(i) and evaluating the resulting equation by Lemmas \ref{lem:tilde_rho} and \ref{lem:E2F2 tilde_rho}.
\end{proof}

\noindent {\it Proof of Theorem \ref{thm:image of natural}(ii).}
Let $\mathcal K$ denote the two-sided ideal of $\H$ generated by $\beta$ and $16 \Omega-24 \alpha+3$.
By (\ref{beta}) the element $\beta\in {\rm Ker}\,\natural$. Using  (\ref{alpha}) and Lemma~\ref{lem:image_omega}(ii) yields that $16 \Omega-24 \alpha+3\in {\rm Ker}\,\natural$. Hence $\mathcal K\subseteq {\rm Ker}\, \natural$. There exists a unique algebra homomorphism $\overline{\natural}:\H/\mathcal K \to \U_{e}$ given by
$$
\overline{\natural}(x+\mathcal K )=x^\natural\qquad\hbox{for all $x\in\H$}.
$$
To see that ${\rm Ker}\, \natural\subseteq \mathcal K$ it suffices to show that $\overline{\natural}$ is injective.
Since $\alpha$ is central in $\H$ the element $\widehat{\Lambda}$ commutes with $\widehat{E^2},\widehat{F^2},\widehat{H}$.
Combined with Lemmas~\ref{lem:bar(HE^2,HF^2)}, \ref{lem:bar(E^2F^2,F^2E^2)} the cosets $\widehat{E^2}+\mathcal K,\widehat{F^2}+\mathcal K,\widehat{\Lambda}+\mathcal K,\widehat{H}+\mathcal K$ satisfy the relations (\ref{Ue1})--(\ref{Ue5}). By Theorem~\ref{thm:Ue_generators} there exists a unique algebra homomorphism $\sharp: \U_{e}\to \H/\mathcal K$ that sends
      \begin{eqnarray*}
      (E^2,F^2,\Lambda,H) &\mapsto &
      (\widehat{E^2}+\mathcal K,\widehat{F^2}+\mathcal K,\widehat{\Lambda}+\mathcal K,\widehat{H}+\mathcal K).
      \end{eqnarray*}

Using (\ref{eq:natural_preimage_E^2})--(\ref{eq:natural_preimage_H}) yields that
\begin{align*}
A&=\frac{\widehat{H}}{4},
\\
B&=\frac{\widehat{E^2}+\widehat{F^2}}{4}-\frac{\widehat{H}^2}{8}+
\frac{\widehat{\Lambda}-1}{4},
\\
C&=\frac{\widehat{E^2}-\widehat{F^2}}{4}.
\end{align*}
Since the algebra $\H$ is generated by $A,B,C$ it follows that the algebra $\H/\mathcal K$ is generated by $\widehat{E^2}+\mathcal K,\widehat{F^2}+\mathcal K,\widehat{\Lambda}+\mathcal K,\widehat{H}+\mathcal K$. By Lemma \ref{lem:bar(E2,F2,L,H)} the algebra homomorphism $\sharp\circ\overline{\natural}$ fixes each of $\widehat{E^2}+\mathcal K,\widehat{F^2}+\mathcal K,\widehat{\Lambda}+\mathcal K,\widehat{H}+\mathcal K$. It follows that $\sharp\circ\overline{\natural}=1$ and thus $\overline{\natural}$ is injective. Theorem \ref{thm:image of natural}(ii) follows.
      \hfill $\square$

%%%%%%%%%%%%%%%%%%%%%%%%%%%%%%%%%%%%%%%%%%%%%%%%%%%%%%%%%%%%%%%%%%%%%%%%%%%%%%%%%%%%%%%%

%%%%%%%%%%%%%%%%%%%%%%%%%%%%%%%%%%%%%%%%%%%%%%%%%%%%%%%%%%%%%%%%%%%%%%%%%%%%%%%%%%%%%%%%

\section{Proof for Theorem \ref{thm:decomposition_Ln} and finite-dimensional irreducible $\U_{e}$-modules}\label{sec:Hmodule}

For any $\U$-module $V$ and any $\theta\in\C$ let
$
V(\theta)=\{v\in V\mid Hv=\theta v\}$.

\begin{prop}\label{prop:V}
  Let $V$ denote a $\U$-module. Then
  $$\bigoplus\limits_{n\in\Z}V(\theta+4n)$$
  is a $\U_{e}$-submodule of $V$ for any $\theta\in\C$.
\end{prop}
\begin{proof}
Let $\theta\in \C$ be given. Clearly $HV(\theta)\subseteq V(\theta)$.
Using (\ref{eq:U_relation_E}) and (\ref{eq:U_relation_F}) yields that $EV(\theta)\subseteq V(\theta+2)$ and $FV(\theta)\subseteq V(\theta-2)$ respectively. It follows that
\begin{gather*}
U_nV(\theta)\subseteq V(\theta+2n)
\qquad
\hbox{for all $n\in\Z$}.
\end{gather*}
By Definition~\ref{defn:Ue} the proposition follows.
\end{proof}

Recall the $\U$-module $L_n$ ($n\in\N$) from Lemma~\ref{lem:Vn_U}.

\begin{defn}\label{defn:Ln0, Ln1}
\begin{enumerate}
  \item For each $n\in\N$ let
  $$
  L_n^{(0)}=\bigoplus_{i\in \Z}L_n(n-4i). $$
  \item For each integer $n\geq 1$ let
  $$
  L_n^{(1)}=\bigoplus_{i\in \Z}L_n(n-4i-2).
  $$
\end{enumerate}
\end{defn}

\begin{lem}
\label{lem:decLn}
For any $n\in \N$ the space $L_n$ is equal to
\begin{align*}
\left\{
\begin{array}{ll}
L_n^{(0)} \qquad &\hbox{if $n=0$},
\\
L_n^{(0)}\oplus L_n^{(1)}
 \qquad &\hbox{if $n\geq 1$}.
\end{array}
\right.
\end{align*}
\end{lem}
\begin{proof}
Immediate from Lemma \ref{lem:Vn_U}(i).
\end{proof}

It follows from Proposition \ref{prop:V} that $L_n^{(0)}$ and $L_n^{(1)}$ are two $\U_{e}$-submodules of $L_n$.

\begin{lem}\label{lem:H-module Ln0}
For each $n\in\N$ the $\U_{e}$-module $L_n^{(0)}$ satisfies the following properties:
\begin{enumerate}
\item There exists a basis $\{u_i^{(0)}\}_{i=0}^{\floor{\frac{n}{2}}}$ for $L_n^{(0)}$ such that
        \begin{align*}
             E^2 u_{i}^{(0)} &=(n-2i+1)(n-2i+2) u_{i-1}^{(0)}
            \quad
            (1\leq i\leq \textstyle{\floor{\frac{n}{2}}}),
            \qquad
            E^2 u_0^{(0)}=0,
            \\
             F^2 u_{i}^{(0)} &=(2i+1)(2i+2) u_{i+1}^{(0)}
            \quad
            (0\leq i\leq \textstyle\floor{\frac{n}{2}}-1),
            \qquad
            F^2 u_{\floor{\frac{n}{2}}}^{(0)}=0,
            \\
             H u_{i}^{(0)} &=(n-4i) u_{i}^{(0)}
            \quad
            (0\leq i\leq \textstyle\floor{\frac{n}{2}}).
        \end{align*}

\item The element $\Lambda$ acts on $L_n^{(0)}$ as scalar multiplication by $\frac{n(n+2)}{2}$.
\end{enumerate}
\end{lem}
\begin{proof}
(i): Recall the basis $\{v_i\}_{i=0}^n$ for $L_n$ from Lemma~\ref{lem:Vn_U}(i).
Observe that
$$
L_n(n-4i)=
\left\{
\begin{array}{ll}
{\rm span}\{v_{2i}\}
\qquad
&\hbox{if $i=0,1,\ldots,\floor{\frac{n}{2}}$},
\\
\{0\}
\qquad
&\hbox{else}
\end{array}
\right.
$$
for all $i\in \Z$.  Let $u_i^{(0)}=v_{2i}$ for $i=0,1,\ldots,\floor{\frac{n}{2}}$. By Definition~\ref{defn:Ln0, Ln1}(i) the $\U_{e}$-module $L_n^{(0)}$ has the basis $\{u_i^{(0)}\}_{i=0}^{\floor{\frac{n}{2}}}$. It is routine to verify the actions of $E^2,F^2,H$ by using Lemma~\ref{lem:Vn_U}(i).

(ii): Immediate from Lemma~\ref{lem:Vn_U}(ii).
\end{proof}

\begin{lem}\label{lem:Ln0 irreducible}
For any $n\in \N$ the $\U_{e}$-module $L_n^{(0)}$ is irreducible.
\end{lem}
\begin{proof}
Let $V$ denote a nonzero $\U_{e}$-submodule of $L_n^{(0)}$.
Note that each eigenspace of $H$ in $L_n^{(0)}$ has dimension one. Since $V$ is a nonzero $H$-invariant subspace of $L_n^{(0)}$ there exists the smallest integer $i$ with $0\leq i\leq \floor{\frac{n}{2}}$ such that $u_i^{(0)}\in V$.

Suppose that $i\geq 1$. Then $E^2 u_i^{(0)}=(n-2i+1)(n-2i+2) u_{i-1}^{(0)}\in V$.
Since $2i\not\in\{n+1,n+2\}$ it follows that $u_{i-1}^{(0)}\in V$, a contradiction. Hence $i=0$. Since $F^{2i} u_0^{(0)}= (2i)! u_i^{(0)}\in V$ for all $i=0,1,\ldots,\floor{\frac{n}{2}}$ it follows that $V=L_n^{(0)}$. The lemma follows.
\end{proof}

\begin{lem}\label{lem:Ln0_nonisomorphic}
The $\U_{e}$-modules $L_n^{(0)}$ for all $n\in \N$ are mutually non-isomorphic.
\end{lem}
\begin{proof}
Suppose that there are $m,n\in \N$ such that the $\U_{e}$-modules $L_m^{(0)}$ and $L_n^{(0)}$ are isomorphic. By Lemma \ref{lem:H-module Ln0}(i) the dimension of $L_m^{(0)}$ is $\floor{\frac{m}{2}}+1$ and the dimension of $L_n^{(0)}$ is $\floor{\frac{n}{2}}+1$. Hence $m\in \{n-1,n,n+1\}$.
It follows from Lemma \ref{lem:H-module Ln0}(ii) that
$m(m+2)=n(n+2)$. Using the above results yields that $m=n$. The lemma follows.
\end{proof}

\begin{lem}\label{lem:H-module Ln1}
For each integer $n\geq 1$ the $\U_{e}$-module $L_n^{(1)}$ satisfies the following properties:
\begin{enumerate}
\item   There exists a basis $\{u_i^{(1)}\}_{i=0}^{\floor{\frac{n-1}{2}}}$ for $L_n^{(1)}$ such that
        \begin{align*}
         E^2 u_{i}^{(1)} &=(n-2i)(n-2i+1) u_{i-1}^{(1)}
        \quad
        (1\leq i\leq \textstyle{\floor{\frac{n-1}{2}}}),
        \qquad
        E^2 u_0^{(1)}=0,
        \\
         F^2 u_{i}^{(1)} &=(2i+2)(2i+3) u_{i+1}^{(1)}
        \quad
        (0\leq i\leq \textstyle\floor{\frac{n-1}{2}}-1),
        \qquad
        F^2 u_{\floor{\frac{n-1}{2}}}^{(1)}=0,
        \\
         H u_{i}^{(1)} &=(n-4i-2) u_{i}^{(1)}
        \quad
        (0\leq i\leq \textstyle\floor{\frac{n-1}{2}}).
\end{align*}
\item   The element $\Lambda$ acts on $L_n^{(1)}$ as scalar multiplication by $\frac{n(n+2)}{2}$.
\end{enumerate}
\end{lem}
\begin{proof}
(i): Recall the basis $\{v_i\}_{i=0}^n$ for $L_n$ from Lemma~\ref{lem:Vn_U}(i). Observe that
$$
L_n(n-4i-2)=
\left\{
\begin{array}{ll}
{\rm span}\{v_{2i+1}\}
\qquad
&\hbox{if $i=0,1,\ldots,\floor{\frac{n-1}{2}}$},
\\
\{0\}
\qquad
&\hbox{else}
\end{array}
\right.
$$
for all $i\in \Z$. Let $u_i^{(1)}=v_{2i+1}$ for $i=0,1,\ldots,\floor{\frac{n-1}{2}}$. By Definition~\ref{defn:Ln0, Ln1}(ii) the $\U_{e}$-module $L_n^{(1)}$ has the basis $\{u_i^{(1)}\}_{i=0}^{\floor{\frac{n-1}{2}}}$. It is routine to verify the actions of $E^2,F^2,H$ by using Lemma~\ref{lem:Vn_U}(i).

(ii): Immediate from Lemma~\ref{lem:Vn_U}(ii).
\end{proof}

\begin{lem}\label{lem:Ln1 irreducible}
For any integer $n\geq 1$ the $\U_{e}$-module $L_n^{(1)}$ is irreducible.
\end{lem}
\begin{proof}
Let $V$ denote a nonzero $\U_{e}$-submodule of $L_n^{(1)}$.
Note that each eigenspace of $H$ on $L_n^{(1)}$ has dimension one. Since $V$ is a nonzero $H$-invariant subspace of $L_n^{(1)}$ there exists the smallest integer $i$ with $0\leq i\leq \floor{\frac{n-1}{2}}$ such that $u_i^{(1)}\in V$.

Suppose that $i\geq 1$. Then $E^2 u_i^{(1)}=(n-2i)(n-2i+1) u_{i-1}^{(1)}\in V$.
Since $2i\not\in\{n,n+1\}$ it follows that $u_{i-1}^{(1)}\in V$, a contradiction. Hence $i=0$. Since $F^{2i} u_0^{(1)}=(2i+1)! u_i^{(1)}\in V$ for all $i=0,1,\ldots,\floor{\frac{n-1}{2}}$ it follows that $V=L_n^{(1)}$. The lemma follows.
\end{proof}

\begin{lem}\label{lem:Ln1_nonisomorphic}
The $\U_{e}$-modules $L_n^{(1)}$ for all integers $n\geq 1$ are mutually non-isomorphic.
\end{lem}
\begin{proof}
Suppose that there are two integers $m,n\geq 1$ such that the $\U_{e}$-modules $L_m^{(1)}$ and $L_n^{(1)}$ are isomorphic. By Lemma \ref{lem:H-module Ln1}(i) the dimension of $L_m^{(1)}$ is $\floor{\frac{m-1}{2}}+1$ and the dimension of $L_n^{(1)}$ is $\floor{\frac{n-1}{2}}+1$. Hence $m\in \{n-1,n,n+1\}$.
It follows from Lemma \ref{lem:H-module Ln1}(ii) that
$m(m+2)=n(n+2)$. Using the above results yields that $m=n$. The lemma follows.
\end{proof}

\begin{thm}\label{thm:Ln0&Ln1 noniso}
The $\U_{e}$-modules $L_n^{(0)}$ for all $n\in\N$ and the  $\U_{e}$-modules $L_n^{(1)}$ for all integers $n\geq 1$ are mutually non-isomorphic.
\end{thm}
\begin{proof}
Pick any two integers $m,n$ with $m\in \N$ and $n\geq 1$. Suppose that the $\U_{e}$-modules $L_m^{(0)}$ and $L_n^{(1)}$ are isomorphic.
By Lemma \ref{lem:H-module Ln0}(i) the dimension of $L_m^{(0)}$ is $\floor{\frac{m}{2}}+1$. By Lemma \ref{lem:H-module Ln1}(i) the dimension of $L_n^{(1)}$ is $\floor{\frac{n-1}{2}}+1$. Hence $m\in \{n-2,n-1,n\}$.
It follows from Lemmas \ref{lem:H-module Ln0}(ii) and \ref{lem:H-module Ln1}(ii) that
$m(m+2)=n(n+2)$. Using the above results yields that $m=n$. By Lemma \ref{lem:H-module Ln0}(i) the element $H$ has the eigenvalue $m$ in $L_m^{(0)}$. However $m$ is not an eigenvalue of $H$ in $L_n^{(1)}$ by Lemma \ref{lem:H-module Ln1}(i), a contradiction. Hence the $\U_{e}$-modules $L_m^{(0)}$ and $L_n^{(1)}$ are non-isomorphic. Combined with Lemmas \ref{lem:Ln0_nonisomorphic} and \ref{lem:Ln1_nonisomorphic} this implies Theorem \ref{thm:Ln0&Ln1 noniso}.
\end{proof}

The proof of Theorem~\ref{thm:decomposition_Ln} is given below.

\medskip

\noindent {\it Proof of Theorem~\ref{thm:decomposition_Ln}.}
Recall from Theorem \ref{thm:image of natural}(i) that ${\rm Im}\,\natural=\U_{e}$. Hence the $\H$-module $L_n^{(0)}$ $(n\in \N)$ is irreducible by Lemma \ref{lem:Ln0 irreducible} and the $\H$-module $L_n^{(1)}$ $(n\geq 1)$ is irreducible by Lemma \ref{lem:Ln1 irreducible}. By Theorem \ref{thm:Ln0&Ln1 noniso} the $\H$-modules $L_n^{(0)}$ for all $n\in \N$ and the $\H$-modules $L_n^{(1)}$ for all integers $n\geq 1$ are mutually non-isomorphic. Theorem~\ref{thm:decomposition_Ln} follows by Lemma \ref{lem:decLn} and the above comments.
\hfill $\square$

\medskip

We finish this section with a classification of the finite-dimensional irreducible $\U_{e}$-modules.

\begin{thm}
\label{thm:Ue-module}
For any $d\in \N$ the $\U_{e}$-modules $L_{2d}^{(0)}$, $L_{2d+1}^{(0)}$, $L_{2d+1}^{(1)}$, $L_{2d+2}^{(1)}$ are all $(d+1)$-dimensional irreducible $\U_{e}$-modules up to isomorphism.
\end{thm}
\begin{proof}
Let $d\in \N$ be given. Suppose that $V$ is a $(d+1)$-dimensional irreducible $\U_{e}$-module. Since $\C$ is algebraically closed and the $\U_{e}$-module $V$ is finite-dimensional, it follows from Schur's lemma that the central element $\Lambda$ acts on $V$ as scalar multiplication by some scalar $\lambda\in \C$. In addition, there exists a scalar $\theta\in \C$ such that $\theta$ is an eigenvalue of $H$ but $\theta+4$ is not an eigenvalue of $H$ in $V$.

Let $w$ denote a $\theta$-eigenvector of $\H$ in $V$. Set
\begin{gather}\label{wi}
w_i=F^{2i} w
\qquad
\hbox{for all $i\in \N$}.
\end{gather}
Applying $w_0$ to either side of Lemma \ref{lem:HE^n,HF^n}(ii) with even $n$ yields that
\begin{gather}\label{Hwi}
H w_i=(\theta-4i) w_i
\qquad
\hbox{for all $i\in \N$}.
\end{gather}
Applying $w_0$ to either side of (\ref{Ue1}) yields that $HE^2w_0=(\theta+4)E^2 w_0$.
By the choice of $\theta$ it follows that
\begin{gather}\label{E2w0}
E^2 w_0=0.
\end{gather}
We apply $w_{i-1}$ ($i\geq 1$) to either side of (\ref{Ue3}) and use the equation (\ref{Hwi}) to evaluate the resulting equation. It follows that
\begin{gather}\label{E2wi}
E^2 w_i\in {\rm span}\{w_{i-1}\}
\qquad
\hbox{for all integers $i\geq 1$}.
\end{gather}
By (\ref{Hwi}) the nonzero vectors among $\{w_i\}_{i\in \N}$ are linearly independent.
Since $V$ is finite-dimensional there are only finitely many indices $i$ such that $w_i\not=0$. Since $w_0\not=0$ there is an integer $d'\geq 0$ such that $w_i\not=0$ for all $i=0,1,\ldots,d'$ and $w_{d'+1}=0$. Let $W$ denote the subspace of $V$ spanned by $\{w_i\}_{i=0}^{d'}$.
By (\ref{wi}) the space $W$ is $F^2$-invariant.
By (\ref{E2w0}) and (\ref{E2wi}) the space $W$ is $E^2$-invariant.
By (\ref{Hwi}) the space $W$ is $H$-invariant.
Hence $W$ is a nonzero $\U_{e}$-submodule of $V$. By the irreducibility of $V$ it follows that $d'=d$. In other words the vectors $\{w_i\}_{i=0}^d$ form a basis for $V$. Note that
\begin{gather}\label{F2wd}
F^2 w_d=0.
\end{gather}

We apply $w_0$ to either side of (\ref{Ue4}) and evaluate the resulting equation by using (\ref{Hwi}) and (\ref{E2w0}). It follows that
\begin{gather}\label{lambda1}
\lambda=\frac{\theta(\theta+2)}{2}
\end{gather}
or
\begin{gather}\label{lambda2}
\lambda=4+\frac{\theta(\theta+6)}{2}.
\end{gather}
We apply $w_d$ to either side of (\ref{Ue3}) and evaluate the resulting equation by using (\ref{Hwi}) and (\ref{F2wd}). It follows that
\begin{gather}\label{lambda3}
\lambda=\frac{(\theta-4d)(\theta-4d-2)}{2}
\end{gather}
or
\begin{gather}\label{lambda4}
\lambda=4+\frac{(\theta-4d)(\theta-4d-6)}{2}.
\end{gather}
Hence there are only the following four possibilities for $V$:

Case 1: The equations (\ref{lambda1}) and (\ref{lambda3}) hold. Solving (\ref{lambda1}) and (\ref{lambda3}) for $\theta$ yields that
\begin{gather}\label{theta-a}
\theta=2d.
\end{gather}
We apply $w_{i-1}$ $(1\leq i\leq d)$ to either side of (\ref{Ue3}) and evaluate the resulting equation by using (\ref{wi}), (\ref{Hwi}), (\ref{lambda1}) and (\ref{theta-a}). It follows that
$$
E^2 w_i=4i(2i-1)(d-i+1)(2d-2i+1) w_{i-1}
\qquad
(1\leq i\leq d).
$$
By Lemma \ref{lem:H-module Ln0} there exists a unique $\U_{e}$-module isomorphism $V\to L_{2d}^{(0)}$ that maps $w_i$ to $(2i)! u_i^{(0)}$ for $i=0,1,\ldots,d$.

Case 2: The equations (\ref{lambda1}) and (\ref{lambda4}) hold. Solving (\ref{lambda1}) and (\ref{lambda4}) for $\theta$ yields that
\begin{gather}\label{theta-b}
\theta=2d+1.
\end{gather}
We apply $w_{i-1}$ $(1\leq i\leq d)$ to either side of (\ref{Ue3}) and evaluate the resulting equation by using (\ref{wi}), (\ref{Hwi}), (\ref{lambda1}) and (\ref{theta-b}). It follows that
$$
E^2 w_i=
4i (2i-1)(d-i+1)(2d-2i+3)
w_{i-1}
\qquad
(1\leq i\leq d).
$$
By Lemma \ref{lem:H-module Ln0} there exists a unique $\U_{e}$-module isomorphism $V\to L_{2d+1}^{(0)}$ that maps $w_i$ to $(2i)! u_i^{(0)}$ for $i=0,1,\ldots,d$.

Case 3: The equations (\ref{lambda2}) and (\ref{lambda3}) hold. Solving (\ref{lambda2}) and (\ref{lambda3}) for $\theta$ yields that
\begin{gather}\label{theta-c}
\theta=2d-1.
\end{gather}
We apply $w_{i-1}$ $(1\leq i\leq d)$ to either side of (\ref{Ue3}) and evaluate the resulting equation by using (\ref{wi}), (\ref{Hwi}), (\ref{lambda2}) and (\ref{theta-c}). It follows that
$$
E^2 w_i=
4i (2i+1)(d-i+1)(2d-2i+1)
w_{i-1}
\qquad
(1\leq i\leq d).
$$
By Lemma \ref{lem:H-module Ln1} there exists a unique $\U_{e}$-module isomorphism $V\to L_{2d+1}^{(1)}$ that maps $w_i$ to $(2i+1)! u_i^{(1)}$ for $i=0,1,\ldots,d$.

Case 4: The equations (\ref{lambda2}) and (\ref{lambda4}) hold. Solving (\ref{lambda2}) and (\ref{lambda4}) for $\theta$ yields that
\begin{gather}\label{theta-d}
\theta=2d.
\end{gather}
We apply $w_{i-1}$ $(1\leq i\leq d)$ to either side of (\ref{Ue3}) and evaluate the resulting equation by using (\ref{wi}), (\ref{Hwi}), (\ref{lambda2}) and (\ref{theta-d}). It follows that
$$
E^2 w_i=
4i (2i+1)(d-i+1)(2d-2i+3)
w_{i-1}
\qquad
(1\leq i\leq d).
$$
By Lemma \ref{lem:H-module Ln1} there exists a unique $\U_{e}$-module isomorphism $V\to L_{2d+2}^{(1)}$ that maps $w_i$ to $(2i+1)! u_i^{(1)}$ for $i=0,1,\ldots,d$.
\end{proof}

%%%%%%%%%%%%%%%%%%%%%%%%%%%%%%%%%%%%%%%%%%%%%%%%%%%%%%%%%%%%%%%%%%%%

%%%%%%%%%%%%%%%%%%%%%%%%%%%%%%%%%%%%%%%%%%%%%%%%%%%%%%%%%%%%%%%%%%%%

%%%%%%%%%%%%%%%%%%%%%%%%%%%%%%%%%%%%%%%%%%%%%%%%%%%%%%%%%%%%%%%%%%%%%%%%%%%%%%%%%%%%%%%%%%%%%%%%%%%%%%%%%%%%%%%%%%%%%%%
\section{The connection to the Terwilliger algebras of $H(D,2)$ and $\frac{1}{2}H(D,2)$}\label{sec:Terwilliger lagebra}
%%%%%%%%%%%%%%%%%%%%%%%%%%%%%%%%%%%%%%%%%%%%%%%%%%%%%%%%%%%%%%%%%%%%%%%%%%%%%%%%%%%%%%%%%%%%%%%%%%%%%%%%%%%%%%%%%%%%%%%

Fix an integer $D\geq 2$. Let $X=X(D)$ denote the
set
$$
\{(x_1,x_2,\ldots,x_D)\,|\,x_1,x_2,\ldots,x_D\in\{0,1\}\}.
$$
For any $x\in X$ let $x_i$ $(1\leq i\leq D)$ denote the $i^{\rm \, th}$ coordinate of $x$.
Recall that the {\it $D$-dimensional hypercube} $H(D,2)$ is a finite simple connected graph with vertex set $X$ and $x,y\in X$ are adjacent if and only if there exists exactly one index $i$ $(1\leq i\leq D)$ such that $x_i\not=y_i$. For any $x,y\in X$ let $\partial(x,y)$ denote the distance from $x$ to $y$ in $H(D,2)$. Note that
$\partial(x,y)=|\{i\,|\, 1\leq i\leq D, x_i\not=y_i\}|$.

The adjacency operator $\A$ of $H(D,2)$ is a linear endomorphism of $\C^X$ given by
\begin{equation*}
\A x=
\sum_{\substack{y\in X\\ \partial(x,y)=1}} y
\qquad
\hbox{for all $x\in X$}.
\end{equation*}
Let $x\in X$ be given. By \cite{BannaiIto1984,TerAlgebraI,TerAlgebraII,TerAlgebraIII}  the dual adjacency operator $\A^*(x)$ of $H(D,2)$ with respect to $x$ is a linear endomorphism of $\C^X$ given by
$$
\A^*(x) y=(D-2\partial(x,y)) y
\qquad\hbox{for all $y\in X$}.
$$
The {\it Terwilliger algebra $\T(x)$ of $H(D,2)$ with respect to $x$} \cite{Hamming:2021,TerAlgebraI,TerAlgebraII,TerAlgebraIII,hypercube2002,Lthypercube,Hamming2006,Huang:CG&Hamming} is the subalgebra of ${\rm End}(\C^X)$ generated by $\A$ and $\A^*(x)$.

\begin{thm}[Theorem 13.2, \cite{hypercube2002}]\label{thm:UtoT(hypercube)}
 For each $x\in X$ there exists a unique algebra homomorphism $\rho(x):\U\to \T(x)$ that sends
  \begin{eqnarray*}
E &\mapsto &
\displaystyle \frac{\A}{2}-\frac{[\A,\A^*(x)]}{4},
\\
F &\mapsto &
\displaystyle \frac{\A}{2}+\frac{[\A,\A^*(x)]}{4},
\\
H &\mapsto & \A^*(x).
\end{eqnarray*}
%Moreover ${\rm Im}\, \rho(x)=\T(x)$ for any $x\in X$.
\end{thm}

\begin{lem}
\label{lem:HtoT(hypercube)}
For each $x\in X$ the algebra homomorphism $\rho(x)\circ \natural:\H\to \T(x)$ maps
  \begin{eqnarray*}
A & \mapsto & \frac{\A^*(x)}{4},
\\
B & \mapsto & \frac{\A^2-1}{4}.
\end{eqnarray*}
\end{lem}
\begin{proof}
Let $x\in X$ be given. Recall $A^\natural$ and $B^\natural$ from (\ref{A}) and (\ref{B}).
By Theorem \ref{thm:UtoT(hypercube)} the image of $A^\natural$ under $\rho(x)$ is as stated.
Using (\ref{U_casimir}) yields that $B^\natural$ is equal to $\frac{(E+F-1)(E+F+1)}{4}$. Hence the image of $B^\natural$ under $\rho(x)$ is as stated. The result follows.
\end{proof}

Let $x\in X$ be given.
By Theorem \ref{thm:UtoT(hypercube)} each $\T(x)$-module is a $\U$-module. We denote by $\C^X(x)$ the natural $\T(x)$-module structure on $\C^X$. Let $V$ denote a vector space. For any integer $n\geq 1$ we write $n\cdot V$ for
$
\underbrace{V\oplus V\oplus \cdots \oplus V}_
{\hbox{{\tiny $n$ copies of $V$}}}$.

\begin{thm}
[Theorem 10.2, \cite{hypercube2002}]
\label{thm:Umodule CX}
For each $x\in X$ the $\U$-module $\C^X(x)$ is isomorphic to
$$
\bigoplus_{k=0}^{\floor{\frac{D}{2}}}
\frac{D-2k+1}{D-k+1} {D\choose k}\cdot L_{D-2k}.
$$
\end{thm}

Define
\begin{align*}
X_e=\left\{
x\in X\,\Big|\, \hbox{$\sum\limits_{i=1}^D x_i$ is even}
\right\}.
\end{align*}
The halved graph $\frac{1}{2}H(D,2)$ of $H(D,2)$ is a finite simple connected graph with vertex set $X_e$ and $x,y\in X_e$ are adjacent if and only if $\partial(x,y)=2$. Hence the adjacency operator of $\frac{1}{2}H(D,2)$ is equal to
\begin{align*}
\frac{\A^2-D}{2}\Big|_{\C^{X_e}}.
\end{align*}
Let $x\in X_e$ be given. By \cite{TerAlgebraI,TerAlgebraII,TerAlgebraIII,BannaiIto1984} the dual adjacency operator of
$\frac{1}{2}H(D,2)$ with respect to $x$ is equal to
\begin{align*}
\left\{
\begin{array}{ll}
%\displaystyle
\frac{1}{2}\A^*(x)\big|_{\C^{X_e}}
\qquad
&\hbox{if $D=2$},
\\
\A^*(x)\big|_{\C^{X_e}}
\qquad
&\hbox{if $D\geq 3$}.
\end{array}
\right.
\end{align*}
Therefore the {\it Terwilliger algebra $\T_e(x)$ of $\frac{1}{2}H(D,2)$ with respect to $x$}  \cite{TerAlgebraI,TerAlgebraII,TerAlgebraIII} is the subalgebra of ${\rm End}(\C^{X_e})$ generated by $\A^2|_{\C^{X_e}}$ and $\A^*(x)\big|_{\C^{X_e}}$.

\begin{thm}
\label{thm:Te(X)}
For each $x\in X_e$ the following hold:
\begin{enumerate}
\item $\T_e(x)=
\{M|_{\C^{X_e}}\,|\, M\in {\rm Im}\,(\rho(x)\circ\natural) \}$.

\item $\T_e(x)=
\{M|_{\C^{X_e}}\,|\, M\in {\rm Im}\,(\rho(x)|_{\U_e})\}$.
\end{enumerate}
\end{thm}
\begin{proof}
(i):  By Definition \ref{defn:H} the algebra $\H$ is generated by $A$ and $B$. Theorem \ref{thm:Te(X)}(i) is now immediate from Lemma \ref{lem:HtoT(hypercube)}.

(ii): Immediate from Theorems \ref{thm:image of natural}(i) and \ref{thm:Te(X)}(i).
\end{proof}

Let $x\in X_e$ be given.
By Theorem \ref{thm:Te(X)} each $\T_e(x)$-module is a $\U_e$-module as well as an $\H$-module. We denote by $\C^{X_e}(x)$ the natural $\T_e(x)$-module structure on $\C^{X_e}$.

\begin{thm}
\label{thm:Uemodule CXe}
For each $x\in X_e$ the $\U_e$-module $\C^{X_e}(x)$ is isomorphic to
\begin{equation*}
\bigoplus_{\substack{k=0 \\ \hbox{\tiny $k$ is even}}}^{\floor{\frac{D}{2}}}\frac{D-2k+1}{D-k+1}\binom{D}{k}\cdot L_{D-2k}^{(0)}
\oplus
\bigoplus_{\substack{k=1 \\ \hbox{\tiny $k$ is odd}}}^{\floor{\frac{D-1}{2}}}\frac{D-2k+1}{D-k+1}\binom{D}{k}\cdot L_{D-2k}^{(1)}.
    \end{equation*}
\end{thm}
\begin{proof}
Let $x\in X_e$ be given. For any $y\in X$ observe that $y\in X_e$ if and only if $\A^*(x)y=(D-4i) y$ for some $i\in \Z$.
It follows from Theorem \ref{thm:UtoT(hypercube)} that
$$
\C^{X_e}=\bigoplus_{i\in \Z}\C^X(D-4i).
$$
Combined with Theorem \ref{thm:Umodule CX} the $\U_e$-module $\C^{X_e}(x)$ is isomorphic to
$$
\bigoplus_{k=0}^{\floor{\frac{D}{2}}}
\frac{D-2k+1}{D-k+1}
{D\choose k}
\cdot
\bigoplus_{i\in \Z} L_{D-2k}(D-4i).
$$
By Definition \ref{defn:Ln0, Ln1} the result follows.
\end{proof}

By \cite{TerAlgebraI} the algebra $\T_e(x)$ is semi-simple.
We remark that the first description of all non-isomorphic irreducible $\T_e(x)$-modules was given in \cite{TerAlgebraIII}.

\begin{cor}
\label{cor:Te(x) structure}
For each $x\in X_e$ the following hold:
\begin{enumerate}
\item The $\T_e(x)$-modules
\begin{align}
&L_{D-2k}^{(0)}
\qquad
\hbox{for all even integers $k$ with $0\leq k\leq \floor{\frac{D}{2}}$};
\label{L_D-2k^0}
\\
&L_{D-2k}^{(1)}
\qquad
\hbox{for all odd integers $k$ with $1\leq k\leq \floor{\frac{D-1}{2}}$}
\label{L_D-2k^1}
\end{align}
are all irreducible $\T_e(x)$-modules up to isomorphism.

\item The algebra $\T_e(x)$ is isomorphic to
\begin{gather*}
\bigoplus_{\substack{k=0 \\ \hbox{\tiny $k$ is even}}}^{\floor{\frac{D}{2}}}
{\rm End}
\left(
\C^{\floor{\frac{D}{2}}-k+1}
\right)
\oplus
\bigoplus_{\substack{k=1 \\ \hbox{\tiny $k$ is odd}}}^{\floor{\frac{D-1}{2}}}
{\rm End}
\left(
\C^{\floor{\frac{D-1}{2}}-k+1}
\right).
\end{gather*}

\item The dimension of $\T_e(x)$ is equal to
$$
{\floor{\frac{D}{2}}+3\choose 3}+
{\ceil{\frac{D}{2}}+1\choose 3}.
$$
\end{enumerate}
\end{cor}
\begin{proof}
(i): Since the $\T_e(x)$-module $\C^{X_e}(x)$ is faithful, each irreducible $\T_e(x)$-module is contained in $\C^{X_e}(x)$. Recall from Lemmas \ref{lem:Ln0 irreducible} and \ref{lem:Ln1 irreducible} that the $\U_e$-modules (\ref{L_D-2k^0}) and (\ref{L_D-2k^1}) are irreducible. Recall from Theorem \ref{thm:Ln0&Ln1 noniso} that the $\U_e$-modules (\ref{L_D-2k^0}) and (\ref{L_D-2k^1}) are mutually non-isomorphic.
Corollary \ref{cor:Te(x) structure}(i) now follows by the above comments and Theorem \ref{thm:Te(X)}(ii).

(ii): By Corollary \ref{cor:Te(x) structure}(i) the algebra $\T_e(x)$ is isomorphic to
$$
\bigoplus_{\substack{k=0 \\ \hbox{\tiny $k$ is even}}}^{\floor{\frac{D}{2}}}
{\rm End}
(L_{D-2k}^{(0)})
\oplus
\bigoplus_{\substack{k=1 \\ \hbox{\tiny $k$ is odd}}}^{\floor{\frac{D-1}{2}}}
{\rm End}
(L_{D-2k}^{(1)}).
$$
Combined with Lemmas \ref{lem:H-module Ln0}(i) and \ref{lem:H-module Ln1}(i) this implies Corollary \ref{cor:Te(x) structure}(ii).

(iii): By Corollary \ref{cor:Te(x) structure}(ii) the dimension of $\T_e(x)$ is equal to
$$
\sum_{\substack{k=0 \\ \hbox{\tiny $k$ is even}}}^{\floor{\frac{D}{2}}}
\left(
\floor{\frac{D}{2}}-k+1
\right)^2
+
\sum_{\substack{k=1 \\ \hbox{\tiny $k$ is odd}}}^{\floor{\frac{D-1}{2}}}
\left(
\floor{\frac{D-1}{2}}-k+1
\right)^2.
$$
The first summation is equal to ${\floor{\frac{D}{2}}+3\choose 3}$ and the second summation is equal to ${\ceil{\frac{D}{2}}+1 \choose 3}$.
\end{proof}

\subsection*{Acknowledgements}
The research is supported by the National Science and Technology Council of Taiwan under the projects
MOST 109-2115-M-009-007-MY2, MOST 110-2115-M-008-008-MY2 and MOST 111-2115-M-A49-005-MY2.

\bibliographystyle{amsplain}
\bibliography{MP}

\end{document}